\documentclass[hidelinks, 12pt,reqno]{amsart}
\usepackage{tikz,dsfont}
\usepackage{amsmath, amssymb, graphicx, mathrsfs, hyperref}
\usepackage{verbatim} 
\usepackage{enumerate}
\newcommand{\F}{\mathbb{F}}

\newtheoremstyle{dotless}{}{}{\itshape}{}{\bfseries}{}{ }{} 
\theoremstyle{dotless}

\newtheorem{theorem}{Theorem}[section]
\newtheorem{lemma}[theorem]{Lemma}

\newtheorem{prop}[theorem]{Proposition}
\newtheorem{conj}[theorem]{Conjecture}

\newtheorem{defn}[theorem]{Definition}
\newtheorem{remark}[theorem]{Remark}

\newtheorem{exam}[theorem]{Example}

\begin{document}

\pagenumbering{arabic} \setcounter{page}{1}

\title{Breaking the 6/5 threshold for sums and products modulo a prime}

\author{G. Shakan, I. D. Shkredov}

\address{Department of Mathematics \\
University of Illinois \\
Urbana, IL 61801, U.S.A.}
\email{shakan2@illinois.edu}

\address{ Steklov Mathematical Institute\\
ul. Gubkina, 8, Moscow, Russia, 119991 \\
and
\\
IITP RAS  \\
Bolshoy Karetny per. 19, Moscow, Russia, 127994\\
and 
\\
MIPT \\ 
Institutskii per. 9, Dolgoprudnii, Russia, 141701\\}
\email{ilya.shkredov@gmail.com}

\begin{abstract}
Let $A \subset \mathbb{F}_p$ of size at most $p^{3/5}$. We show $$|A+A| + |AA| \gtrsim |A|^{6/5 + c},$$ for $c = 4/305$. Our main tools are the cartesian product point--line incidence theorem of Stevens and de Zeeuw and the theory of higher order energies developed by the second author.
\end{abstract}

\maketitle
\providecommand{\keywords}[1]{\textbf{\textit{Keywords---}} #1}
\section{Introduction}

For a finite subset $A$ of an abelian group, we define the sumset and product set via $$A+A :=\{a+b : a \in A, b \in B\}, \ \ \ AA := \{ab : a\in A , b\in B\}.$$
The Erd\H{o}s--Szemer\'edi sum--product conjecture \cite{ES} states the following. We adopt Vinogradov's notation and say $a \ll b$ when $a =O(b)$. 

\begin{conj}\label{esconj}\cite{ES} Fix $\delta < 1$. For any finite $A \subset \mathbb{Z}$, one has
 $$|A+A| + |A A| \gg |A|^{1 + \delta}.$$
\end{conj}

The current state of the art progress towards Conjecture~\ref{esconj} is due to the first author \cite{Sha} building upon the works of \cite{El, So, BW, KS, KS2,Shr3, RSS}. The underlying idea in these works is a set with few sums and products to create point--line incidence structure with many incidences. This is in tension with point--line incidence bounds, for instance the well--known Szemer\'edi--Trotter theorem \cite{ST}, valid over $\mathbb{R}$.

It is expected that Conjecture~\ref{esconj} is true over $\mathbb{F}_p$ as long as $|A|$ is not too large. The first general sum--product bound in $\mathbb{F}_p$ was due Bourgain, Katz, and Tao \cite{BKT} in which they also gave the first general point--line incidence bound in $\mathbb{F}_p^2$. Since then, sum--product estimates have found applications to RIP matrices \cite{Bo, Bo2}, exponential sums \cite{BGK,Bo3}, and finite field Fourier restriction \cite{RuSh} to name just a few. Their sum--product estimate was quantitatively improved in \cite{Ga, KaSh,BG,Ru1}. Roche--Newton, Rudnev, and the second author \cite{RRS} obtained a big quantitative improvement coming from the breakthrough point--plane incidence bound of Rudnev \cite{Ru}.
\begin{theorem}\cite{RRS}\label{RRS} Let $A \subset \mathbb{F}_p$ of size at most $p^{5/8}$. Then $$|A+A| + |AA| \gg |A|^{6/5}.$$
\end{theorem}

We are able to improve upon Theorem~\ref{RRS}. We say $b \gtrsim a$ if $a = O(b \log^c |A|)$ for some $c >0$.

\begin{theorem}[Sum--product bound]\label{main} Let $A \subset \mathbb{F}_p$ of size at most $p^{3/5}$. Then
$$|A \pm A|^{35} |AA|^{26} \gtrsim |A|^{74}.$$ In particular,  $$|A\pm A| + |AA| \gtrsim |A|^{6/5 + 4/305}.$$
\end{theorem}
Thus we improve upon Theorem~\ref{RRS} by more than $1/100$. One can replace products with ratios in Theorem~\ref{main}. Also, we can switch the role of plus and times in our arguments, which allows us to interchange sums and products in Theorem~\ref{main}. For $|A| \geq p^{3/5}$, see for instance, \cite{Ga2}. The proof of Theorem~\ref{main} is entirely self--contained modulo the point--line incidence bound of Stevens and de Zeeuw \cite{SZ}. 

We describe our two--step approach to Theorem~\ref{main}. First, we use the point--line incidence bound of Stevens and Zeeuw \cite{SZ} to bound the fourth order additive energy in terms of the product set. An application of Cauchy--Schwarz at this step recovers Theorem~\ref{RRS}. We then bound this fourth order energy via Theorem~\ref{main2} below to conclude the sumset is slightly larger than this application of Cauchy--Schwarz. In the second step we utilize the theory of higher energies developed by the second author \cite{SS, Sh1,Sh2}. We emphasize that this is a particularly simple instance of the operator method, which we describe in detail. 

We actually give two proofs that one can improve Theorem~\ref{RRS}. The first is elementary, though quantitatively worse. We hope the similarities seen in these two proofs will further highlight the flexibility of the operator method. 

In Section~\ref{Notation} we define quantities related to higher energies, provide some basic properties and examples, and introduce Theorem~\ref{main2}. In Section~\ref{incidence} we apply incidence geometry to bound a fourth order energy in terms of the product set. In Section~\ref{elementary}, we provide an elementary improvement to Theorem~\ref{RRS}. In Section~\ref{operator} we introduce the notion of operators and provide some basic properties and examples. In Section~\ref{optimized} we finish the proof of Theorem~\ref{main}. Note that Section~\ref{elementary} and Section~\ref{optimized} do not rely on each other in a fundamental way.

\section{Notation and Set--up}\label{Notation}

 Let $G$ be an abelian group and $A$ and $B$ be two finite subsets. We define a representation function $$r_{A-B}(x) := \#\{(a,b) \in A \times  B : x = a - b\} , \ \ x \in G .$$ Thus $$\sum_{x \in G} r_{A -B}(x) = |A||B|,$$ and so $r_{A-B}(x)$ is $|A||B| |A-B|^{-1}$ on average. The distribution of $r_{A-B}(x)$ depends on the additive structure of $A$ and $B$. To understand this better, it is convenient to introduce the {\em additive energy} of $A$ and $B$ is defined via 

$$E^{+}(A,B) := \sum_x r_{A-B}(x)^2 = \sum_x r_{A+B}(x)^2 ,\ \ \ E^{+}(A) := E^{+}(A,A) .$$ We have the trivial bounds $|A|^2 \leq E^+(A) \leq |A|^3$. Heuristically, $E^+(A)$ is closer to $|A|^3$ when $A$ has additive structure. This is seen more clearly by the relation $$E^+(A) = \#\{(a , b , c, d) \in A^4 : a +b = c+ d\}.$$ Quantitatively, we have the following application of Cauchy--Schwarz \begin{equation}\label{CS} |A|^4 \leq E^+(A) |A+A| . \end{equation} 
We now introduce the $k^{\rm th}$ order energy (we are primarily interested in $k = 2,4$): $$ E_k^{+}(A,B) := \sum_x r_{A-B}(x)^k, \ \ \ E_k^{+}(A) := E_k^{+}(A,A) .$$ 
We take $k=4$, and in this case the trivial bounds are $|A|^4 \leq E_4^+(A) \leq |A|^5$. As with the additive energy, $E_4^+(A)$ being close to $|A|^5$ is an indication of additive structure.
There is the following relation to additive energy via H\"{o}lder \begin{equation}\label{holder} E^+(A,B) = \sum_x r_{A-B}(x)^{4/3 + 2/3} \leq E_4^+(A,B)^{1/3} |A|^{2/3} |B|^{2/3} .\end{equation} 
In particular, 
$$
	E_4^{+} (A,B) \ge \frac{|A|^4 |B|^4}{|A+B|^3} \geq |A||B|,
$$
and so taking $B = \pm A$, we find the analog of  \eqref{CS} for $E^+_4(A)$ is \begin{equation}\label{E4} |A|^8 \leq E_4^+(A) |A\pm A|^3 \,.\end{equation} 
Even with the optimal information $E_4^+(A) \leq |A|^4$, we can only conclude with \eqref{E4} that $|A \pm A| \geq |A|^{4/3}$.
 We have the relation $$E_4^+(A) = \#\{(a , b , c, d,e,f,g,h ) \in A^8 : a -b = c- d = e -f = g-h\}.$$ 
Geometrically, as in \cite[Equation 9]{MRSS}, $E_4^+(A)$ counts the number of collinear quadruples in $A \times A$ that lie on a line of the form $y = x + d$. By double counting the projections onto the coordinate axes, we find that \begin{equation}\label{quad} E_4^+(A) = \sum_{x,y,z} |A \cap (A+x) \cap (A+y) \cap (A+z)|^2.\end{equation} Indeed, $a , a+x , a+y , a+z \in A$ can be extended to a collinear quadruple in $A \times A$ on a line of slope one of the form  $$(a,b) , (a+x , b') , (a+y , b'') , (a+ z , b''') , \ \ \ b , b' , b'' , b''' \in A,$$ if and only if $b = b' + x = b'' + y = b''' + z $. One natural question is to relate $E_4^+(A)$ to lower order energies. We give one possible way to do this in Lemma~\ref{lem1} below.

We will also need the following definition, which is crucial in the proof of Theorem~\ref{main}. 
\begin{defn}\label{dplus}
Let $A \subset G$ be  finite. We define  $$d_4^{+}(A) := \sup_{B \neq \emptyset}\frac{ E_4^+(A,B) }{|A| |B|^3}.$$ \end{defn} 
It follows from Definition~\ref{dplus} that $$1\leq E_4^+(A) |A|^{-4} \leq d_4^{+}(A) \leq |A|.$$ All of the identities we mentioned for $E_4^+(A)$ can be generalized in a natural way for $E_4^+(A,B)$. Intuitively, the closer $d_4^{+}(A)$ is to $|A|$, the more additive structure $A$ has. Indeed $d_4^+(A)$ is a fourth moment analog of the additive energy of a set, where we are allowed flexibility in choosing $B$. The next two key examples highlight the subtle difference in working with $d_4^+(A)$ as opposed to $E_4^+(A)$. We take $x \asymp y$ to mean $x \ll y$ and $x \gg y$.

\begin{exam}[Small progression and many generic elements] Let $n,m \geq 1$ be fixed integers. Let $$A= \{1 , \ldots , n\} \cup \{n , \ldots , n^m\} , \ \ \ n^5 \asymp m^4$$ be
	 the union 
	 of an arithmetic progression with a disassociated set.  Thus $|A| \asymp m$. We have $$E_4^+(A)|A|^{-4} \asymp 1.$$On the other hand, letting $B = \{1 , \ldots , n\}$, we find $$E_4^+(A,B) |A|^{-1} |B|^{-3} \asymp n^2 m^{-1} \asymp m^{3/5} \asymp |A|^{3/5}.$$ Thus $d_4^+(A) \gg |A|^{3/5}$, which is significantly larger than $E_4^+(A) |A|^{-4}$. \end{exam}

The previous example is important in the paper of the second author \cite{Sh2}. In the next example we will see the quantity $d_4^+(A)$ is large because it is contained in a set with additive structure, rather than containing a set with additive structure in the previous example. 

\begin{exam}[Dense subset of interval]\label{example} Consider a random $A \subset \{1 , \ldots , n\}$ where each element is chosen independently and uniformly with probability $p > n^{-1/2}$. We have $|A| \sim pn$ with high probability. Let $B = \{1 , \ldots , n\}$. It follows from Chernoff's inequality (for instance, Chapter 1 of \cite{TV}) and the union bound, that every $x$ with $r_{B-B}(x) \geq  p^{-2} \log n$ satisfies
$$r_{A-B}(x) \sim p \cdot r_{B-B}(x) , \ \ \ \  r_{A-A}(x) \sim p^2 \cdot  r_{B-B}(x) .$$ 
This implies $$\frac{ E_4^+(A,B) }{|A| |B|^3} \sim p^2|A|, \ \  \frac{ E_4^+(A) }{|A|^4} \sim p^3|A|  .$$ In this example $d_4^+(A)$ is larger than what is predicted by the fourth order energy, where we only allow $B=A$ in Definition \ref{dplus}. One can show $d_4^+(A) \sim p^2 |A|$. 
\end{exam}

We note that $B$ cannot be too large or small in Definition~\ref{dplus}.

\begin{remark}\label{size} Observe that the supremum in Definition \ref{dplus} is achieved for some $$ |B| \leq |A|^{3/2}.$$ Indeed for $|B| \geq |A|^{3/2}$, we have $$\frac{E_4^+(A,B) }{|A||B|^3} \leq \frac{|A|^3}{|B|^2} \leq 1 \leq d^+(A).$$ 
\end{remark}
From \eqref{E4} we have the following Cauchy--Schwarz \begin{equation}\label{CS4} |A|^4 \leq |A\pm A|^3 d_4^+(A) . \end{equation} 

So \eqref{CS4} is \eqref{E4} for the quantity $d_4^+(A)$. One of our goals will be to improve upon \eqref{CS4}, utilizing the flexibility in the choice of $B$ in Definition~\ref{dplus}.
\begin{theorem}[Small energy implies large sumset]\label{main2} Let $A$ be a subset of an abelian group. Then $$|A\pm A| \gtrsim |A|^{48/35} /d_4^+(A)^{13/35} $$
\end{theorem}

To see that this improves \eqref{CS4}, one just needs to check $48/35 > 4/3$. Theorem~\ref{main2} has limitations. In light of Example~\ref{example}, the best one can hope for is $$|A+A| \gtrsim |A|^{3/2} d_4^+(A)^{-1/2}.$$ This unproven bound, combined with Proposition~\ref{energyreduction} below, would give an Elekes \cite{El} type bound for the sum--product problem over $\mathbb{F}_p$. 

\section{Reduction to an energy estimate via incidence geometry}\label{incidence}

The main purpose of this section is to show if the product set of $A \subset \mathbb{F}_p$ is small then $d_4^+(A)$ is small. 

\begin{prop}[Incidence bound]\label{energyreduction} Let $A \subset \mathbb{F}_p$ such that $|A| < p^{3/5}$. Then $$d_4^+(A) \lesssim |AA|^2 |A|^{-2}.$$
\end{prop}
The key input in the proof of Proposition~\ref{energyreduction} is the point--line incidence bound of Stevens and de Zeeuw \cite{SZ}. Theorem~\ref{main} then follows immediately upon combining Proposition~\ref{energyreduction} and Theorem~\ref{main2}. On the other hand, combining Proposition~\ref{energyreduction} with \eqref{CS4} recovers Theorem~\ref{RRS}.

\begin{remark} Proposition~\ref{energyreduction} is why we are required to introduce higher order energies of the form in Definition~\ref{dplus}. We provide some context. This idea originated from a paper of Schoen and the second author \cite{SS}. We say a finite $A \subset \mathbb{R}$ is convex if it has increasing consecutive differences. It was known for convex sets that  $$\max_{B \neq \emptyset} \frac{E_3^+(A,B)}{|A|^{} |B|^{2}} \lesssim 1.$$ They were able to use this to show that $$|A+A| \gtrsim |A|^{3/2 + \delta}, \ \ \ \text{for some }   \delta > 0,$$ thus breaking the ``3/2--barrier."
\end{remark}

We will use an incidence theorem specialized for cartesian products due to Stevens and de Zeeuw \cite[Theorem 4]{SZ}, building work in  \cite{BKT, Jo}. The exact statement we use is a slight modification of it appearing in work of Murphy and Petridis (see also \cite[Lemma 12]{Shr5}).


\begin{theorem}\label{SZ} \cite[Theorem 7]{MP}
    Let $A,B \subseteq \F_p$ be sets. Let $\mathcal{P} = A\times B$ and $ \mathcal{L}$ be a collection of lines in $\F^2_p$.
    Then 
    \begin{equation}
    \mathcal{I}(\mathcal{P}, \mathcal{L}) \ll |A|^{3/4} |B|^{1/2} |\mathcal{L}|^{3/4} + |\mathcal{L}| + |A| |B|  + \frac{|A| |B| |\mathcal{L}|}{p} \,.
    \end{equation}  
\end{theorem}

Theorem~\ref{SZ} is most naturally interpreted as a fourth moment estimate, that is a $d_4^+(A)$ estimate. One can see this by the exponent $3/4$ on $\mathcal{L}$ in the conclusion. This differs from the real case, in which Szemer\'edi--Trotter gives an exponent $2/3$, which is most naturally a third moment estimate. This was taken advantage of in recent work of the first author \cite{Sha} to improve the Balog--Wooley decomposition \cite{BW}.

\begin{proof}[Proof of Proposition~\ref{energyreduction}]
Let $B$ be a set obtaining the maximum in Definition~\ref{dplus}. By a dyadic decomposition, there is a $\tau \geq 1$ such that $$E_4^+(A,B) \lesssim |D_{\tau}| \tau^4 ,  \ \ D_{\tau} = \{x \in A-B : r_{A-B}(x) \geq \tau\}.$$ Thus to prove Proposition~\ref{energyreduction}, it is enough to show \begin{equation}\label{enough} |D_{\tau}| \tau^4 |A|^{-1} |B|^{-3} \ll |AA|^{2} |A|^{-2}. \end{equation} We plan to apply Theorem~\ref{SZ} to the point set $$D_{\tau} \times AA$$ and the $|B||A|$ lines of the form (without loss of the generality one can suppose  that $0\notin A$)
$$\{(u , v) \in \mathbb{F}_p^2 : v = \frac{u}{r} - b \}, \ \ \ r \in A , \ b \in B.$$ 
Observe that for any $a \in A$ there are $|A|$ solutions to the equation $$a = q/r , \ \ q \in AA , \ r \in A.$$ Applying Theorem~\ref{SZ}, we obtain that

\begin{align*}
|A||D_{\tau}| \tau &\leq |A| \#\{(x , a , b) \in D_{\tau} \times A \times B : x = a -b\} \\
& \leq \#\{(x , u , r, b) \in D_{\tau} \times AA \times A  \times B : x = u/r -b\} \\
& \ll |D_{\tau}|^{3/4} |AA|^{1/2} |A|^{3/4} |B|^{3/4} + |A| |B| + |D_{\tau}| |AA| + |A||B||D_{\tau}||AA| p^{-1}.
\end{align*}
Thus $|A||D_{\tau}| \tau$ must be at most a constant times one of the terms on the right hand side. 
In the main case, that is the first term in the bound dominates, we have
$$ |A| \tau |D_{\tau}| \ll |D_{\tau}|^{3/4} |AA|^{1/2} |A|^{3/4} |B|^{3/4},$$ which after simplifying gives \eqref{enough}.

We now handle the error terms. We will use the following two inequalities repeatedly \begin{equation}\label{Dtau} |D_{\tau}| \tau \leq |A||B| ,\end{equation}
\begin{equation}\label{tau} \tau \leq |A| , |B|. \end{equation}
Suppose $$|A||D_{\tau}| \tau \ll |A| |B|.$$ Then we have $$ |D_{\tau}| \tau \ll |B|,$$ which combined with \eqref{tau} gives \eqref{enough} via  $$\frac{|D_{\tau}| \tau^4}{|A| |B|^3} \ll \frac{\tau^3 |B|}{|A| |B|^3} \ll 1 \leq |AA|^2 |A|^{-2}.$$
Next, we suppose $|A||D_{\tau}| \tau  \ll |D_{\tau}| |AA|$ and so $|A|\tau \ll |AA|$. Combining with \eqref{tau} and \eqref{Dtau}, we find
$$ \frac{|D_{\tau}| \tau^4}{|A| |B|^3} \ll |AA|^2 |A|^{-2} \frac{ |D_{\tau}| \tau^2 }{|A| |B|^3} \leq |AA|^2 |A|^{-2} |B|^{-1},$$ which is better than \eqref{enough}.

We may assume \begin{equation}\label{AA} |AA| \leq |A|^{3/2}, \end{equation} for otherwise \eqref{enough} follows immediately from $|D_{\tau}| \tau^4 |A|^{-1} |B|^{-3} \leq |A|$.

Suppose that $|A||D_{\tau}| \tau \ll  |A||B||D_{\tau}||AA| p^{-1}$ and so $p \tau \ll |B| |AA|$. We may suppose $\tau^3 \geq |B|^2 |AA|^2 |A|^{-2}$, since otherwise we may use \eqref{Dtau} to obtain \eqref{enough} via $$|D_{\tau}| \tau^4 |A|^{-1} |B|^{-3} \leq \tau^3 |B|^{-2} \leq |AA|^{2} |A|^{-2}.$$ Then
 $$|AA|^2|A|^{-2} |B|^2p^3 \leq \tau^3 p^3 \ll |B|^3 |AA|^3,$$ 
which simplifies to $$p^3 \ll |AA| |B||A|^2 \leq |A|^5,$$ by \eqref{AA} and Remark~\ref{size}. This case cannot happen by our assumption $|A| < p^{3/5}$ in Proposition~\ref{energyreduction}.

Thus in all possible cases \eqref{enough} holds and so Proposition~\ref{energyreduction} follows.

\end{proof}

We end this section with a summary of the proof of Proposition~\ref{energyreduction}. If $d_4^+(A)$ is large, then the equation $$x = a - b , \ \ x \in D , \  a \in A, \  b\in B,$$ has many solutions for some choices of $D$ and $B$. On the other hand, if $|AA|$ is small, then we may efficiently write $a = q/r$ (an important case to keep in mind is when $A$ is a multiplicative subgroup). Thus there are many solutions to $x = q/r - b$, which is in contrast with the Stevens and de Zeeuw point--line incidence bound.

\section{Elementary energy estimates}\label{elementary}

Our main goal is to use elementary methods, adopted from \cite{SS}, to improve upon \eqref{CS4}, which combined with Proposition~\ref{energyreduction} improves upon the sum--product estimate in Theorem~\ref{RRS}. Wo only handle the ``minus" case, since the ``plus" case is a bit more technically involved (see \cite[Theorem 1]{SS}).
 
We need following notation for $x \in G$, 
\begin{equation}\label{translate} A_x : = A \cap (A+x).\end{equation} Thus $$|A_x| = r_{A - A}(x).$$

The following lemma allows us to pass from fourth order energy to additive energy of subsets of the form \eqref{translate}. This is a special case of \cite[Lemma 2.8]{SV}.
\begin{lemma}\label{lem1} Let $A$ be a subset of an abelian group $G$. Then $$E_4^+(A) = \sum_{x , w \in G} E^+(A_x , A_w).$$ 
\end{lemma}

\begin{proof}
By \eqref{quad}, we have \begin{align*}E_4^+(A) &= \sum_{x,y,z} |A \cap (A+x) \cap (A+y) \cap (A+z)|^2 \\
& = \sum_{x ,y,z} |A_x \cap (A_{z -y} + y)|^2 
 = \sum_{x, w,y} |A_x \cap (A_w + y)|^2 \\
& = \sum_{x,w,y}  r_{A_x - A_w}(y)^2 = \sum_{x , w} E^+(A_w , A_x)\end{align*}
\end{proof}

The next lemma asserts if the fourth energy of $A$ is small, then many of the sumsets from \eqref{translate} are large. This is a special case of \cite[Lemma 2.6]{SV}.

\begin{lemma}\label{energy}Let $P$ be any subset of $A-A$. Then $$\left(\sum_{x \in P} |A_x| \right)^4 \leq E_4^+(A)\sum_{x,w \in P} |A_x \pm A_w|  .$$
\end{lemma}

\begin{proof} The proof is Lemma~\ref{lem1} and two applications of  Cauchy--Schwarz, the first of which being \eqref{CS}. We have \begin{align*} \left(\sum_{x \in P} |A_x| \right)^4 & = \left(\sum_{x ,w \in P} |A_x||A_w| \right)^2 \leq  \left(\sum_{x ,w \in P} |A_x \pm A_w|^{1/2} E^+(A_x , A_w)^{1/2} \right)^2 \\
& \leq \sum_{x ,w \in P} |A_x \pm A_w| \sum_{x,w}  E^+(A_x , A_w) = E_4^+(A)  \sum_{x ,w \in P} |A_x \pm A_w|
\end{align*}
\end{proof}

We remark that we could apply Lemma~\ref{energy} with $P = A-A$ along with the trivial bound $$|A_x - A_w| \leq |A - A|,$$ to obtain $$|A|^8 \leq E_4^+(A) |A-A|^3.$$ Combining this with the following implication of Definition~\ref{dplus} \begin{equation}\label{triv} E_4^+(A) \leq d_4^+(A) |A|^4, \end{equation} we are lead back to the bound \eqref{CS4}. We improve upon this argument with Katz--Koester inclusion \cite{KK}. 

\begin{prop}\label{diff} Let $A$ be a finite subset of an abelian group $G$. Then $$|A-A| \gg \frac{|A|^{15/11}}{d_4^+(A)^{4/11}}.$$
\end{prop} 

This is an improvement on \eqref{CS4}. Combining this with Proposition~\ref{energyreduction} gives the difference--product estimate $$|A-A| + |AA| \gg |A|^{6/5 + 1/95}.$$ 

\begin{proof} Let $$P = \{x \in A -A : r_{A-A}(x) \geq |A|^2 / (2 |A-A|) \},$$ be a set of popular differences. Then $$\sum_{x \in P} r_{A-A}(x) \geq |A|^2 / 2.$$ By Lemma~\ref{energy} and \eqref{triv} we find \begin{equation}\label{inter}|A|^8 \ll E_4^+(A) \sum_{x ,w \in P} |A_x - A_w| \leq d_4^+(A) |A|^4 \sum_{x ,w \in P} |A_x - A_w|  .\end{equation} 
Let $D= A-A$. Then Katz--Koester inclusion yields $$A_x - A_w \subset D \cap (D+x) \cap (D - w) \cap (D+ x - w).$$ Thus we have \begin{align*}\sum_{x ,w \in P} |A_x - A_w| & \leq \sum_{x ,w \in P}  |D \cap (D+x) \cap (D -w) \cap (D+ x -w)| \\ & = \sum_{x , w \in P} |D_x \cap (D_x  -w)| \leq \sum_{x \in P} |D_x|^2.\end{align*}
Combining this with \eqref{inter}, we have  \begin{equation}\label{setversion} |A|^4 d_4^+(A)^{-1} \leq E^+(D) \end{equation}
Now the trivial bound is $$\sum_{x \in P} |D_x|^2 \leq |D|^3,$$ which upon combing with \eqref{setversion} would yield \eqref{CS4}. We can improve upon this when $d_4^+(A)$ is small. Using $x \in P$ and H\"{o}lder via \eqref{holder}, we find \begin{align*} \sum_{x \in P} |D_x|^2 & \leq |D| \sum_{x \in P} r_{D-D}(x) \ll \frac{|D|^2}{|A|^2} \sum_{x \in P} r_{D-D}(x) r_{A-A}(x) \leq\frac{|D|^2}{|A|^2}  E^+(A,D) \\ & \leq \frac{|D|^2}{|A|^2} E_4^+(A, D)^{1/3} (|A| |D|)^{2/3} \leq \frac{|D|^2}{|A|^2} (d_4^+(A) |A| |D|^3)^{1/3} (|A| |D|)^{2/3} \end{align*}
We combine this with \eqref{setversion} to obtain $$ \frac{|A|^{4}}{d_4^+(A)} \leq d_4^+(A)^{1/3} |A|^{-1} |D|^{11/3}.$$ Simplifying gives Proposition~\ref{diff}.

\end{proof}

\begin{remark} In the proof of Proposition~\ref{diff}, we used the identity $$E^+(A,D) = \sum_{x} r_{A-A}(x) r_{D-D}(x),$$ which is an consequence of $$a - a' = d  - d' \ \  \text{if and only if}  \ \ a - d = a' - d'.$$ There is no such identity for higher energies.
\end{remark}

\section{Operator method basics}\label{operator}

We now introduce the notions from the theory of operators and higher order  energies, developed by the second author \cite{Sh1, Sh2}. We do not require much, but we provide some context. This theory was already used to improve sum--product estimates over the real numbers \cite{KS, KS2, Sha}. The results we need are much simpler than the real case, for a reason we will explain below. We will use some basic facts from linear algebra, which can be found in \cite{HJ}.

 For two functions $f,g :G \to \mathbb{C}$ with finite support, we define the convolution of $f$ and $g$ via \begin{equation*}\label{f:convolutions}
(f*g) (x) := \sum_{y\in G} f(y) g(x-y)\,, \ \ \ (f \circ g) (x) := \sum_{y \in G} f(y) g(x+y).
\end{equation*}
Thus $$r_{A+B}(x) = 1_A * 1_B(x) , \ \ \ r_{A-B}(x) = 1_A \circ 1_B(x).$$ 
We first want to illustrate why the study of operators is a natural one in additive combinatorics. To this end, we interpret the set $A$ as a {\em convolution operator}, $K_A$ or $\widetilde{K}_A$, via $$K_A(f)(x) := 1_A * f (x) , \ \ \ \widetilde{K}_A(f)(x) := 1_A \circ f (x).$$ All of the quantities we have previously defined can be given in terms of $\widetilde{K}_A$: \begin{align*}E^+(A) &= ||\widetilde{K}_A( 1_A)||_2^2 , \ E_4^+(A) = ||\widetilde{K}_A (1_A)||_4^4, \\ ||\widetilde{K}_A(1_B)||_1 &= |A||B|, \ |A-B| =| {\rm supp}( \widetilde{K}_A (1_B)| \\ d_4^+(A) &= |A|^{-1} ||\widetilde{K}_A||_{\ell^4 \to \ell^{4/3}}^3 , \  r_{A-B}(x) = \widetilde{K}_A (1_B)(x) .\end{align*} Indeed $\widetilde{K}_A(1_0) = 1_A$, and so we can recover $A$ from $\widetilde{K}_A$. Thus there is no loss in generality in considering $\widetilde{K}_A$ in place of $A$. We also need a {\em local} version of $\widetilde{K}_A$, given by $$\widetilde{K}_A^A (f)(x) : = 1_A(x) \widetilde{K}_A(f)(x).$$ Again $\widetilde{K}_A^A(1_0)(x) = 1_A(x)$, so we may recover $A$ from $\widetilde{K}_A^A$. A priori, $\widetilde{K}_A^A$ is not a finite dimensional operator, but from the definition we see \begin{equation}\label{KA} \widetilde{K}_A^A(1_z) (x)= 1_A(x)1_{A-z}(x).\end{equation} This can be nonzero only if $z \in A-A$ and so we may interpret $\widetilde{K}_A^A$ as a $A \times (A-A)$ matrix with the $(x,z)$ entry given in \eqref{KA}.

Given a function $g : G \to \mathbb{C}$, we let $T_A^g$ and $\widetilde{T}_A^g$ be the $|A| \times |A|$ square matrices defined by $$T_A^g(x,y) := g(x-y) 1_A(x) 1_A(y), \ \ \ \widetilde{T}_A^g(x,y) := g(x+y) 1_A(x) 1_A(y).$$ One important example is $g(z) = r_{A-A}(z)$. From \eqref{KA}, we see that $$T_A^{r_{A-A}} = \widetilde{K}_A^A (\widetilde{K}_A^A)^*,$$ where $(\widetilde{K}_A^A)^*$ is the adjoint of $\widetilde{K}_A^A$. Thus $T_A^{r_{A-A}}$ is a non-negative, self adjoint operator. 

\begin{exam}Let us compute $T_A^{r_{A-A}}$ two examples. First take $S$ to be a Sidon set. Then $$T_S^{r_{S-S}} = (|S| - 1) I + J,$$ where $J$ is the all ones matrix. Now, suppose that $H$ is a subgroup of $G$. In this case $$T_H^{r_{H-H}} = |H| J.$$ The eigenvalues for $T_S^{r_{S-S}}$ are $$2|S| -1 , |S| -1 , \ldots , |S|-1,$$ while the eigenvalues for $T_H^{r_{H-H}}$ are $$|H|^2, 0 , \ldots , 0.$$ Thus the eigenvalues of $T_A^{r_{A-A}}$ for sets with additive structure are localized. 
\end{exam}

One can check that adding a few disassociated elements to $H$ in the previous example does not greatly alter the eigenvalues. Before the next example, we recall some facts from linear algebra. Let $M$ be an $n \times n$ Hermitian matrix with eigenvalues $|\mu_1| \geq \ldots \geq |\mu_n|$. Then $$|\mu_1| = \sup_{||v||_2 = 1} |\langle Mv , v \rangle|,  \ \ \ {\rm trace}(M^k) = \sum_{j=1}^n \mu_j^k.$$

\begin{exam} Let us give a basic demonstration on how to use the operator $T_A^{r_{A-A}} $ to obtain higher energy bounds. Note that $T_A^{r_{A-A}}$ is self adjoint and non--negative and so it has $|A|$ real eigenvalues $\mu_1 \geq \cdots \geq \mu_{|A|} \geq 0$. Thus $$\mu_1^2 \leq \sum_{j=1}^{|A|} \mu_j^2 \leq {\rm trace}({(T_A^{r_{A-A}}})^2) = \sum_{x,y \in A} r_{A-A}(x-y) r_{A-A}(y-x) = E_3^+(A).$$ On the other hand, we may take the unit vector $v : = |A|^{-1/2} (1 , \ldots , 1)^T$ and see that $$\mu_1 \geq \langle T_A^{r_{A-A}} v , v \rangle = E^+(A) |A|^{-1}.$$ Putting this together, we obtain $$E^+(A)^2 \leq E_3^+(A) |A|^2,$$ which recovers an application of Cauchy--Schwarz. The situation that arises after will be more complicated combinatorially (we take a fourth power of a matrix), but is in a similar spirit to this example. 
\end{exam}

We now introduce a different choice of $g$. Computing higher powers of $T_A^g$ and $\widetilde{T}_A^g$ for this choice of $g$ gives rise to other combinatorial quantities. \begin{remark}\label{Pidentity} Let $P \subset A+A$ and consider $\widetilde{T}_A^{1_P}$.  Note that $\widetilde{T}_A^{1_P}$ is self--adjoint, but not non--negative in general. Let $|\mu_1| \geq \dots \geq |\mu_{|A|}|$ be the eigenvalues. For the unit vector $v = |A|^{-1/2}(1 , \ldots , 1)$, we have $$|\mu_1| \geq \langle \widetilde{T}_A^{1_P} v , v \rangle = |A|^{-1} \sum_{x , y  \in A} 1_P(x+y) =  |A|^{-1} \sum_{z \in P} r_{A+A}(z).$$  Thus for any integer $k \geq 1$, we have 
\begin{equation}\label{trace} \left( |A|^{-1} \sum_{z \in P} r_{A+A}(z) \right)^{2k} \leq  \mu_1^{2k} \leq \sum_{j=1}^{|A|} \mu_j^{2k} \leq  {\rm trace}((\widetilde{T}_A^{1_P})^{2k}). \end{equation} 
The right hand side of \eqref{trace} can be written as $$\sum_{x_1 , \ldots , x_{2k} \in A} 1_P(x_1+x_2) \cdots 1_P(x_{2k-1} + x_{2k}) 1_P(x_{2k} + x_1).$$ See Lemma~\ref{eigen} below for one possible estimate of this quantity.
\end{remark}

For a finite set $A \subset G$ we define another natural notion of higher order energy via \begin{align*} T_k(A) & = \#\{(a_1 , \ldots , a_{2k} )\in A^{2k} : a_1 + \ldots + a_k = a_{k+1} + \ldots + a_{2k}\} \\ &= \int_{\mathbb{R}/ \mathbb{Z}} |\sum_{a \in A} e^{2 \pi i a \theta} |^{2k} d \theta.\end{align*} 
Higher powers of $\widetilde{T}_A^{1_P}$ in \eqref{trace} give rise combinatorial quantities involving higher order energies of $A$.

\begin{lemma}\cite[Proposition 31]{Sh2}\label{eigen} Let $k\geq 2$ be a power of two and $A, B \subset G$ be two finite sets. Then for any $P \subset A - B$, one has 

$$\left(\sum_{x\in P} r_{A-B}(x)\right)^{4k} \leq |A|^{2k} |B|^{2k} T_k(P) \sum_x r_{A-A}(x)^k r_{B-B}(x)^k.$$
\end{lemma}

We only need a special case of Lemma~\ref{eigen}, and we will provide a proof in this special case in Lemma~\ref{energy} below. We state Lemma~\ref{eigen} it here to illustrate the generality of the operator method. The case $B = A$ is easier as one can prove Lemma~\ref{eigen} with the spectral theory of self--adjoint operators rather than singular value decomposition. In the case $B= A$, $P=A-A$, and $k=2$, Lemma~\ref{eigen} reduces to \begin{equation}\label{first} |A|^8 \leq E_4^+(A) E^+(P).\end{equation} We show that if $A$ has a small product set than we have non--trivial bounds for both quantities on the right hand side of \eqref{first}. Actually we already saw this inequality in \eqref{setversion}, but we need Lemma~\ref{eigen} to obtain a {\em statistical version}, where $P$ is allowed to be a set of popular differences.

\section{Optimized energy estimates}\label{optimized}

Our main goal is to prove Theorem~\ref{main2} and improve upon \eqref{CS4}. The proof requires two separate parts. The first is a simple application of the operator method developed by the second author. This is a special case of Lemma~\ref{eigen}, but we provide a proof here.

\begin{lemma}\label{energy}\cite[Proposition 31]{Sh2}
	Let $A$ be a subset of an abelian group and $P\subseteq A-A$. 
	Then
$$
	\left( \frac{\sum_{x\in P} r_{A-A} (x)}{|A|}\right)^{8} \leq E^{+}_{4} (A) E^+ (P) \,.
$$ 
	Similarly, for any $P\subseteq A+A$ the following holds 
$$
\left( \frac{\sum_{x\in P} r_{A+A} (x)}{|A|}\right)^{8} \leq E^{+}_{4} (A) E^{+} (P) \,.
$$
\end{lemma}

The idea is to bound the trace of the fourth power of a matrix in two ways: {\em spectrally} and {\em combinatorially}. 

Lemma~\ref{energy} should be compared with \eqref{setversion}. This is a ``statistical version" of a set inequality, which the operator method is suitable for. Actually, one can modify the application of Katz--Koester given above to prove Lemma~\ref{energy} for difference sets, but it is unclear how to do so for sumsets.

We will use the following identity, valid for finite $P$, \begin{equation}\label{energyP} \sum_{\alpha , \beta , \gamma} 1_P(\alpha) 1_P(\alpha - \beta) 1_P(\gamma - \beta) 1_P(\gamma) = E^+(P).\end{equation}
Indeed, $$\sum_{\alpha} 1_P(\alpha) 1_P(\alpha - \beta) = r_{P-P}(\beta) , \ \ \ \sum_{\gamma} 1_P(\gamma - \beta) 1_P(\gamma) = r_{P-P}(\beta).  $$

\begin{proof}
We prove the plus statement, the minus being easier. Consider the matrix $M = \widetilde{T}_A^{1_P}$ defined by $$M(x,y) = 1_A(x) 1_A(y) P(x+y).$$  Note that $M$ is self adjoint and so all of its eigenvalues are real. By Remark~\ref{Pidentity}, we have that the absolute value of the largest eigenvalue is at least $|A|^{-1} \sum_{z \in P} r_{A+A}(z)$, and so the trace of $M^4$ is at least $$\left( |A|^{-1} \sum_{z \in P} r_{A+A}(z)\right)^4.$$
On the other hand the trace of $M^4$ is equal to \begin{equation}\label{trace2}\sum_{x,y,z,w \in A} 1_P(x + y) 1_P(y+z) 1_P (z+ w) 1_P(w+x).\end{equation}
We change variables, leaving $x$ unchanged, to $$\alpha = x + y, \beta = x - z , \gamma = x + w .$$ Thus $x \in A \cap (\alpha - A) \cap (\beta +A) \cap (  \gamma - A)$ and so we find that \eqref{trace2} is

\begin{align*} =&\sum_{\alpha, \beta, \gamma} 1_P(\alpha) 1_P(\alpha - \beta) 1_P (\gamma - \beta) 1_P(\gamma)|A \cap (\alpha - A) \cap (\beta + A) \cap (\gamma - A)| \\
\leq & \left( \sum_{\alpha, \beta, \gamma} 1_P(\alpha) 1_P(\alpha - \beta) 1_P (\gamma - \beta) 1_P(\gamma) \sum_{\alpha, \beta, \gamma} |A \cap (\alpha - A) \cap (\beta + A) \cap (\gamma - A)|^2 \right)^{1/2}  \\ 
=& E(P)^{1/2}  \left( \sum_{\alpha, \beta, \gamma} |A \cap (\alpha - A) \cap (\beta + A) \cap (\gamma - A)|^2 \right)^{1/2} \end{align*} using Cauchy--Schwarz and then \eqref{energyP}. Now we want to use \eqref{quad}, but there is a sign issue. It turns out we still can. In a similar manner to Lemma~\ref{lem1}, recalling the definition in \eqref{translate}, we have 
\begin{align*} & \sum_{\alpha, \beta, \gamma} |A \cap (\alpha - A) \cap (\beta + A) \cap (\gamma - A)|^2 = \sum_{\beta ,\alpha , \gamma} |A_{\beta} \cap (-A_{\gamma - \alpha} + \alpha)|^2 \\ &=
\sum_{\beta, w} E^+(A_{\beta} , - A_{w}) =  \sum_{ \beta , w} E^+( A_{\beta} , A_w) \\ &   = \sum_{\alpha , \beta , \gamma} |A \cap (A + \alpha) \cap (A + \beta) \cap (A + \gamma)|^2 \end{align*}
Thus by \eqref{quad}, we find that \eqref{trace2} is $\leq E_4^+(A)^{1/2}E^+(P)^{1/2}$, as desired.
\end{proof}

\begin{remark} In the sum--product problem over $\mathbb{R}$, third order energies are the object of study. In this case, it would be advantageous to modify Lemma~\ref{energy} by taking a third power of $M$. One is presented with the difficulty that the third powers of the eigenvalues can be negative and one cannot bound the trace of $M^3$ from below by the cube of the largest eigenvalue, as in \eqref{trace}. This makes the analysis harder (see \cite[Theorem 11]{Sh4}).
\end{remark}

We are now ready to prove Theorem~\ref{main}. Before we do so, let us reevaluate our position after Lemma~\ref{energy}. Taking $P = A+A$ in Lemma~\ref{energy}, and using \eqref{triv} we find that $$|A|^4 d_4^+(A)^{-1} \leq E^+(P) \leq |A+A|^3,$$ where we used the trivial bound $E^+(P) \leq |P|^3$ in the last inequality. Simplifying, we are lead back to \eqref{CS4}. Any nontrivial improvement over this trivial bound will give us an improvement over \eqref{CS4}. There is not a unique way to proceed here, as we saw in Section~\ref{elementary}.  We are interested in the best quantitative bounds so we present the following argument. The idea is to write $P +P= (A+A) + P = A + (A+P)$, and then bound this by taking $B = A+P$ in Definition~\ref{dplus}. We then bound the size of $A+P$ by using Definition~\ref{dplus} again with $B = P$. 
\begin{proof}[Proof of Theorem~\ref{main2}:]
We only prove the plus version, the minus version following in a similar manner. Let $$P = \{x \in A+A : r_{A+A}(x) \geq 2^{-1} |A|^2 |A+A|^{-1}\}.$$ Thus $$\sum_{x\in P} r_{A+A}(x) \geq |A|^2 / 2.$$We apply Lemma~\ref{energy} and \eqref{triv} to obtain $$|A|^4 d_4^+(A)^{-1} \ll E^+(P).$$
By our choice of $P$, we have $$|A|^2 |A+A|^{-1}r_{P+P}(x) \ll  r_{A+A+P}(x),$$ and so 
\begin{equation}\label{inter2}E^+(P) \ll \frac{|A+A|^2}{|A|^4} \sum_x r_{A+A + P}(x)^2.\end{equation}
By a dyadic decomposition, we find a $\Delta \geq 1$ such that
$$ \sum_x r_{A+A + P}(x)^2 \lesssim \Delta^2 \sum_x r_{A+T}(x)^2 ,  \ \ \ T = \{x : \Delta \leq r_{A+P}(x)  \leq 2 \Delta \} . $$
Thus \eqref{inter2} is $$\ll  \frac{|A+A|^2}{|A|^4}\Delta^2 \sum_x r_{A+T}(x)^2 \leq \frac{|A+A|^2}{|A|^4}\Delta^2 E_4^+(A,B)^{1/3} |A|^{2/3} |T|^{2/3} ,$$ by H\"{o}lder's inequality as in \eqref{holder}. By Definition~\ref{dplus}, this is $$\ll \frac{|A+A|^2}{|A|^4}\Delta^2 d_4^+(A)^{1/3} |A| |T|^{5/3} = \frac{|A+A|^2}{|A|^4}d_4^+(A)^{1/3} |A| (\Delta |T|)^{14/9} (\Delta^4 |T|)^{1/9}.$$
Now from the definition of $T$, we have that $$|T| \Delta^4 \leq \sum_x r_{A +P}(x)^4 \leq d_4^+(A) |A| |P|^3, \ \ |T| \Delta \leq \sum_x r_{A+P}(x) \leq |A||P|\,, $$ and so putting this altogether we obtain $$|A|^4 d_4^+(A)^{-1} \ll E^+(P) \lesssim |A+A|^2 d_4^+(A)^{4/9} |A|^{-4/3} |P|^{17/9}.$$ Using $P \subset A+A$ and simplifying gives Theorem~\ref{main2}.

\end{proof}

It is interesting to note that the techniques in Section~\ref{elementary} can be modified to give the bound of Theorem~\ref{main2} for difference sets. We do not see how one can do this for the sumsets.

\section*{Acknowledgements}

The first author is partially supported by NSF grant  DMS--1501982 and would
like to thank Kevin Ford for financial support. This work was completed while the first author enjoyed the hospitality of JKU Linz and was partially supported by the Austrian Science Fund, Project F5507--N26, which is part of the Special Research Program, ``Quasi--Monte Carlo Methods: Theory and Applications." The first author also thanks the participants of the Georgia Discrete analysis conference as well as the NSF--CBMS Conference on Additive Combinatorics from a Geometric Viewpoint for enlightening discussions. The authors are particularly grateful to Sophie Stevens. The authors also thank Oliver Roche--Newton and Audie Warren for clarifying discussions.

\end{document}